\documentclass[12pt]{amsart}
\usepackage[utf8]{inputenc}
\usepackage[margin=1.15in]{geometry}
\usepackage{amsmath, amssymb, tikz, tikz-cd, amsthm, hyperref, enumitem, stmaryrd, bm,mathdots}

\title{Integral Zariski dense surface groups in $\SL(n,\R)$}
\author{Michael Zshornack}
\address{Department of Mathematics, University of California Santa Barbara, Santa Barbara, CA 93106, USA.}
\email{\href{mailto:zshornack@math.ucsb.edu}{zshornack@math.ucsb.edu}}
\urladdr{\href{https://sites.google.com/view/michael-zshornack}{https://sites.google.com/view/michael-zshornack}}

\newtheorem{theorem}{Theorem}[section]
\newtheorem{corollary}{Corollary}[theorem]
\newtheorem{lemma}[theorem]{Lemma}

\theoremstyle{remark}
\newtheorem*{remark}{Remark}

\DeclareMathOperator{\PSL}{PSL}

\DeclareMathOperator{\SL}{SL}
\DeclareMathOperator{\GL}{GL}
\DeclareMathOperator{\SO}{SO}
\DeclareMathOperator{\Sp}{Sp}

\DeclareMathOperator{\Hom}{Hom}

\DeclareMathOperator{\Hit}{Hit}

\newcommand{\C}{\mathbf{C}}

\newcommand{\R}{\mathbf{R}}
\newcommand{\Z}{\mathbf{Z}}
\newcommand{\Q}{\mathbf{Q}}
\renewcommand{\H}{\mathbf{H}}
\renewcommand{\O}{\mathcal{O}}
\newcommand{\F}{\mathbf{F}}

\begin{document}

\begin{abstract}
    Given a number field $K$, we show that certain $K$-integral representations of closed surface groups can be deformed to being Zariski dense while preserving many useful properties of the original representation. This generalizes a method due to Long and Thistlethwaite who used it to show that thin surface groups in $\SL(2k+1,\Z)$ exist for all $k$.
\end{abstract}

\maketitle

\section{Introduction}
Thin groups are a class of groups which have a rich number of arithmetic properties that have been of interest in recent years. Given a finitely generated subgroup $\Gamma<\GL(n,\Z)$, we let $G$ denote its Zariski closure: $G:=\operatorname{Zcl}(\Gamma)$. We say that the group $\Gamma$ is \textit{thin} if it is of infinite index in $G(\Z)$ (see \cite{sarnak2012notes}). Tools used to study these groups, such as Super Approximation, expander families and the affine sieve, have been rich areas of research that have seen a lot of progress in the last 10--15 years (see \cite{BourgainGamburd2008} and \cite{Bourgain2010}).

One large class of examples of these come from free subgroups which are thin in $\SL(n,\R)$. Finding such subgroups is fairly well understood (see, for example, \cite{fuchsrivinthin}). Much less well understood is in subgroups which are thin in $\SL(n,\R)$ and do not decompose as a free product of two smaller subgroups. In this paper, we seek to further understand ways in which one might construct freely indecomposable thin subgroups of $\SL(n,\R)$ through methods of low-dimensional topology.

We let $S=S_g$ denote a connected, closed, orientable surface of genus $g\geq 2$. Our goal will be to use techniques available to us from higher Teichm\"uller theory to construct discrete and faithful representations of $\pi_1(S)$ which are Zariski dense in $\SL(n,\R)$. 


The starting point of this construction is to look at points on the \textit{Hitchin component} of $S$, which we denote by $\Hit_n(S)$. These are components of the $\PSL(n,\R)$-character variety of $S$, $\mathcal{R}_n(S)$, which contain the image of Teichm\"uller space under the map $\mathcal{T}(S)\to\mathcal{R}_n(S)$ given by composing a Fuchsian representation $\rho_0:\pi_1(S)\to\PSL(2,\R)$ with the irreducible representation $\tau_n:\PSL(2,\R)\to\PSL(n,\R)$. These components were studied by Hitchin in \cite{hitchin92} and later by Labourie in \cite{Labourie2006} where the authors showed that the representations on these components had many of the same properties which were classical facts about representations in Teichm\"uller space (which we briefly summarize in \S\ref{sec:HitchinRepsProperties}). In particular, these properties of the Hitchin component ensure that the representations we construct will be discrete, faithful and can be lifted to $\SL(n,\R)$, thus providing a method for constructing subgroups of the form we seek.

This method of producing Zariski dense surface subgroups of $\SL(n,\R)$ originates in \cite{longthis:2020} where the authors proved the existence of surface subgroups of $\SL(2k+1,\Z)$ for arbitrary $k\geq 1$ whose image was Zariski dense in $\SL(2k+1,\R)$. The starting point for their construction is the representation coming from the hyperbolic structure on the orbifold with signature $S^2(3,4,4)$ (which is unique up to conjugacy due to these orbifolds being rigid). Direct calculation shows that the image of the representation coming from the holonomy of this hyperbolic structure can be conjugated into $\SL(n,\Z)$ after composing with $\tau_n$ for $n$ odd. The authors took this representation and utilized a bending construction of Thurston to deform it so that, after passing to a surface subgroup of finite index, they were left with a representation of a surface group into $\SL(2k+1,\Z)$ with Zariski dense image. Various tools in the theory of Hitchin representations and the theory of algebraic groups were needed in their proof to ensure this bending construction could be done while preserving certain properties of the original representation, but this result provided the first examples of freely indecomposable isomorphism classes of thin subgroups of $\SL(n,\R)$ for infinitely many $n$.

In this paper, we seek to generalize the methods used in \cite{longthis:2020} to show this process of bending an integral representation to being Zariski dense can also be done in a much more general context to produce more examples of Zariski dense subgroups of $\SL(n,\R)$. For any subring $R\subseteq\R$, we will call a representation $\rho:\pi_1(S)\to\SL(n,\R)$ an $R$-point of the character variety if it can be conjugated in $\SL(n,\R)$ to have image in $\SL(n,R)$. Throughout the rest of this paper, we will let $K/\Q$ be a number field with ring of integers $\O_K$. We assume that $K\subset\R$ (so in particular, $K$ is not totally imaginary) and that $K$ has class number one. The class number one hypothesis is mainly for technical reasons and can perhaps be removed in future work, but our main result is the following.

\begin{theorem}
\label{maintheorem}
For any $n\geq 2$, if $\Hit_n(S)$ contains an $\O_K$ point, then there is a finite sheeted cover $\widetilde{S}\to S$ so that $\Hit_n(\widetilde{S})$ contains an $\O_K$-point whose image is Zariski dense in $\SL(n,\R)$.
\end{theorem}
\begin{remark}
When $K\neq \Q$, we prove an even stronger result (Theorem \ref{maintheorem:numberfield}) that circumvents the need to pass to a finite sheeted cover. 
\end{remark}

At its core, the proof of Theorem \ref{maintheorem} is a construction of a similar nature to the one presented in \cite{longthis:2020} with some key modifications made at various steps to handle the slightly more general setup. In addition to handling the even $n$ case, this theorem utilizes a slightly different version of the bending construction more suited to the topology of $S$. These modifications also lend themselves to some other simplifications in the proof which we will discuss later.

When we specialize Theorem \ref{maintheorem} to the case where $K=\Q$, we will be left with a representation of $\pi_1(S)$ into $\SL(n,\Z)$ with Zariski dense image in $\SL(n,\R)$. The image also is of infinite index in $\SL(n,\Z)$ due to results of Margulis. Namely, since surface groups certainly surject onto $\Z$, the image will automatically be of infinite index in $\SL(n,\Z)$ hence thin (see \cite{margulisnormalsubgroups}). This representation will be on the Hitchin component, hence faithful, and so the image will define a freely indecomposable thin sugroup of the form we wanted. Therefore, Theorem \ref{maintheorem} leads to the following useful corollary on the existence of freely indecomposable thin subgroups of $\SL(n,\R)$.

\begin{corollary}
If $\Hit_n(S)$ contains a $\Z$-point, then $\SL(n,\Z)$ contains a thin surface subgroup.
\end{corollary}

Thus the problem of finding non-free thin subgroups of $\SL(n,\R)$ can be reduced to finding integral representations of surface groups on Hitchin components. Of course, finding such representations still remains a difficult problem and at the time of writing, examples of these are only known in the case where $n=2k+1$ is odd or $n=4$.

\subsubsection*{Acknowledgments}
The author would like to thank Darren Long for the guidance provided throughout the writing of this paper. In particular, his numerous helpful discussions held with the author and edits of early drafts were instrumental in the completion of this work.

\section{Properties of Hitchin representations}
\label{sec:HitchinRepsProperties}
The input to the construction used to prove Theorem \ref{maintheorem} is a representation $\rho:\pi_1(S)\to\PSL(n,\O_K)$ lying on the Hitchin component of $\mathcal{R}_n(S)$. We start with representations on such components because they possess many useful properties akin to the representations in $\mathcal{T}(S)$.

\begin{theorem}[Labourie, \cite{Labourie2006}]
\label{labourie:purelyloxodromic}
If $\rho:\pi_1(S)\to\PSL(n,\R)$ belongs to the Hitchin component, it is discrete, faithful, purely loxodromic and strongly irreducible.
\end{theorem}

In this context, $\rho$ is \textit{purely loxodromic} if and only if for every non-identity element $\gamma\in\pi_1(S)$, $\rho(\gamma)$ is diagonalizable with real eigenvalues, all with distinct absolute values. We also say $\rho$ is \textit{strongly irreducible} if the restriction of $\rho$ to any finite index subgroup of $\pi_1(S)$ is irreducible.

Alessandrini, Lee and Schaffhauser have also generalized the conclusions of Theorem \ref{labourie:purelyloxodromic} to also hold in the case where $X$ is a compact connected $2$-dimensional orbifold with boundary and $\chi(X)<0$. Their result (Theorem 2.28 in \cite{ALSorbifolds}) provides many of the same properties for non-closed surfaces which will be necessary for the bending construction of \S\ref{section:bending}.

The main utility of this theorem is that the bending process we use to deform $\rho$ preserves the path component $\rho$ lies on in $\mathcal{R}_n(S)$. Thus as $\rho$ originally lied on the Hitchin component, then so will the bent representation, hence it will still be discrete, faithful and purely loxodromic.

We may also circumvent the (mild) annoyance of having to work with projective matrices due to the fact that representations in $\mathcal{T}(S)$ can be lifted to $\SL(2,\R)$ and the fact that the Hitchin component is connected. The following fact is known, but we record a proof here as well.

\begin{lemma}
Let $\rho:\pi_1(S)\to\PSL(n,\R)$ be a representation on the Hitchin component. Then $\rho$ admits a lift $\widetilde{\rho}:\pi_1(S)\to\SL(n,\R)$. Moreover, for any nontrivial $\gamma\in\pi_1(S)$, the lift $\widetilde{\rho}$ can be chosen so that $\widetilde{\rho}(\gamma)$ has distinct positive real eigenvalues.
\end{lemma}
\begin{proof}
We start with the following presentation of $\pi_1(S)$:
\[
\pi_1(S)=\langle a_1,b_1,\ldots,a_g,b_g:[a_1,b_1]\ldots[a_g,b_g]=1\rangle.
\]
Following \cite{Goldman1988}, fix lifts $\widetilde{\rho}(a_i)$ and $\widetilde{\rho}(b_i)\in\SL(n,\R)$ for each $i=1,\ldots,g$. The question of whether or not these choices define a representation $\widetilde{\rho}:\pi_1(S)\to\SL(n,\R)$ depends on the value of the element
\[
[\widetilde{\rho}(a_1),\widetilde{\rho}(b_1)]\ldots[\widetilde{\rho}(a_g),\widetilde{\rho}(b_g)]\in\left\{\pm I_n\right\}.
\]
Thus, the obstruction to lifting $\rho\in\Hom(\pi_1(S),\PSL(n,\R))$ to $\widetilde{\rho}\in\Hom(\pi_1(S),\SL(n,\R))$ defines a map
\[
o:\Hom(\pi_1(S),\PSL(n,\R))\to\{\pm I_n\}.
\]
This map must be locally constant as small deformations of $\rho$ don't change the value of the above produce of commutators of lifts and $\rho$ is a Hitchin representation, it lies on the same component as the representations coming from Teichm\"uller space. Therefore, it suffices to show that a discrete and faithful representation $\rho:\pi_1(S)\to\PSL(2,\R)$ lifts to $\SL(2,\R)$.

In this case, any representation $\rho:\pi_1(S)\to\PSL(2,\R)$ defines a circle bundle over $S$ given by $(\widetilde{S}\times S^1)/\rho(\pi_1(S))$, where the action on $S^1$ is given by the natural action $\rho$ induces on the circle at infinity of $\H^2$. In this case, the representation $\rho:\pi_1(S)\to\PSL(2,\R)$ lifts to $\SL(2,\R)$ if and only if the Euler number of this circle bundle is divisible by $2$. When $\rho$ is a discrete and faithful representation, the circle bundle $\rho$ defines is topologically equivalent to the unit tangent bundle of $S$, which has Euler number $\pm(2g-2)$, and so $\rho$ will indeed lift to some $\widetilde{\rho}:\pi_1(S)\to\SL(2,\R)$.

We can also guarantee that the lifts can be chosen so that any nontrivial element has distinct positive real eigenvalues as follows. Fix some nontrivial $\gamma\in\pi_1(S)$. Since the eigenvalues of $\rho(\gamma)$ in $\PSL(n,\R)$ are distinct, real and non-zero by \ref{labourie:purelyloxodromic}, the set
\[
\left\{\rho:\pi_1(S)\to\PSL(n,\R):\begin{tabular}{c}
all eigenvalues of a lift of $\rho(\gamma)$\\
to $\SL(n,\R)$ have the same sign
\end{tabular}
\right\}
\]
is both open and closed in the Hitchin component. But then again, when $\rho$ is Fuchsian, we can write $\rho=\tau_n\rho_0$ for $\rho_0:\pi_1(S)\to\PSL(2,\R)$ in $\mathcal{T}(S)$. Here, if $\rho_0(\gamma)$ has eigenvalues $\lambda$ and $\lambda^{-1}$, $\rho(\gamma)$ will have eigenvalues $\lambda^{n-2k+1}$ for $1\leq k\leq n$. Thus for the lifts of Fuchsian representations, their eigenvalues have all the same sign. So the above set is open and closed and nonempty, so must be all of the Hitchin component. Thus, for any given $\gamma\in\pi_1(S)$, we can always find a lift $\widetilde{\rho}:\pi_1(S)\to\SL(n,\R)$ so that the eigenvalues of $\widetilde{\rho}(\gamma)$ have the same sign. If all the eigenvalues happened to be negative, then we must be in the case where $n$ is even, in which case $-\widetilde{\rho}$ is also a lift of $\rho$ into $\SL(n,\R)$ with the eigenvalues of $-\widetilde{\rho}(\gamma)$ all distinct and positive.
\end{proof}

Therefore, for the rest of this paper, we may assume our starting representation on the Hitchin component takes the form $\rho:\pi_1(S)\to\SL(n,\O_K)$. For convenience as well, whenever we fix a nontrivial $\gamma\in\pi_1(S)$, we will assume we have chosen the lift of the representation so that $\rho(\gamma)$ then also has distinct, positive real eigenvalues. Aside from just notational convenience, these lifted representations will still be discrete and faithful due to being on the Hitchin component and so their images will define non-free subgroups of $\SL(n,\O_K)$ of the form we wish to study.

From this point, the goal will be to take the representation $\rho$ and perform a series of deformations so that, after possibly passing to a finite sheeted cover, we are left with a representation of a surface group into $\SL(n,\O_K)$ which is Zariski dense in $\SL(n,\R)$. The following theorem will be crucial to establishing this result.

\begin{theorem}[Guichard, \cite{guichard}]
\label{guichard}
Let $\rho:\pi_1(S)\to\SL(n,\R)$ be a lift of a representation on the Hitchin component and let $G$ be the Zariski closure of $\rho(\pi_1(S))$. Then $G$ is conjugate to one of the following groups:
\begin{enumerate}[label=(\roman*)]
    \item $\tau_n(\SL(2,\R))$.
    \item $\Sp(2k,\R)$ if $n=2k$ is even.
    \item $\SO(k+1,k)$ if $n=2k+1$ is odd.
    \item The image of the $7$-dimensional fundamental representation of the short root of $G_2$.
    \item $\SL(n,\R)$.
\end{enumerate}
\end{theorem}
\begin{remark}
An alternate proof of this same classification is also given in \cite{sambarino2020infinitesimal}.
\end{remark}

We will assume $n\neq 7$ for simplicity so that the possible Zariski closures will only be conjugate to options (i)--(iii) or (v) in the above list. From here, we will show that if the Zariski closure of $\rho(\pi_1(S))$ is one of (i)--(iii) in the above list, then there is a deformation of $\rho$ with the properties that it is: $K$-integral, on the Hitchin component and has a strictly larger Zariski closure. By performing this deformation multiple times if necessary, we will be able to guarantee that the Zariski closure at the final stage will be all of $\SL(n,\R)$ hence define a Zariski dense subgroup. Guichard's classification of possible Zariski closures is also used in this manner in \cite{longthis:2020}, but also relies on a generalization of the result to representations of orbifold fundamental groups on Hitchin components found in \cite{ALSorbifolds}. 

\section{Bending around nonseparating curves}
\label{section:bending}
We now describe the bending construction which will be used to deform the representations in order to successively enlarge the Zariski closures of the image. For our surface $S=S_g$, we fix a nonseparating simple closed curve $\gamma$. We let $S'$ denote the compact surface one gets by cutting $S$ along $\gamma$ and denote by $\gamma_1$ and $\gamma_2$ the two boundary components of $S'$. Up to some homeomorphism of $S$, we can suppose our setup is of the form in Figure \ref{fig:SS'}.

\begin{figure}[ht]
    \centering

    \tikzset{every picture/.style={line width=0.75pt}} 

    \begin{tikzpicture}[x=0.75pt,y=0.75pt,yscale=-1,xscale=1]

        \draw  [color={rgb, 255:red, 0; green, 0; blue, 255 }  ,draw opacity=1 ] (292.37,151.35) .. controls (292.37,140.31) and (295.73,131.35) .. (299.87,131.35) .. controls (304.02,131.35) and (307.37,140.31) .. (307.37,151.35) .. controls (307.37,162.4) and (304.02,171.35) .. (299.87,171.35) .. controls (295.73,171.35) and (292.37,162.4) .. (292.37,151.35) -- cycle ;
        \draw  [color={rgb, 255:red, 0; green, 0; blue, 255 }  ,draw opacity=1 ] (294.71,148.37) -- (292.13,153.33) -- (289.71,148.29) ;
        \draw  [color={rgb, 255:red, 0; green, 0; blue, 255 }  ,draw opacity=1 ] (292.65,78.75) .. controls (292.65,67.71) and (296.01,58.75) .. (300.15,58.75) .. controls (304.29,58.75) and (307.65,67.71) .. (307.65,78.75) .. controls (307.65,89.8) and (304.29,98.75) .. (300.15,98.75) .. controls (296.01,98.75) and (292.65,89.8) .. (292.65,78.75) -- cycle ;
        \draw  [color={rgb, 255:red, 0; green, 0; blue, 255 }  ,draw opacity=1 ] (290.01,80.68) -- (292.49,75.67) -- (295.01,80.66) ;
        \draw [color={rgb, 255:red, 0; green, 0; blue, 255 }  ,draw opacity=1 ]   (555.08,118.08) .. controls (556.11,101.67) and (599.22,100.33) .. (600,115) ;
        \draw   (400,115) .. controls (400,79.1) and (444.77,50) .. (500,50) .. controls (555.23,50) and (600,79.1) .. (600,115) .. controls (600,150.9) and (555.23,180) .. (500,180) .. controls (444.77,180) and (400,150.9) .. (400,115) -- cycle ;
        \draw    (435,115) .. controls (436.29,120) and (452.29,120.86) .. (455,115) ;
        \draw    (540,115) .. controls (541.29,120) and (557.29,120.86) .. (560,115) ;
        \draw    (439.94,118.08) .. controls (442.37,114.08) and (447.51,114.37) .. (449.94,118.08) ;
        \draw    (545.08,118.08) .. controls (547.51,114.08) and (552.65,114.37) .. (555.08,118.08) ;
        \draw [color={rgb, 255:red, 0; green, 0; blue, 255 }  ,draw opacity=1 ] [dash pattern={on 0.84pt off 2.51pt}]  (555.08,118.08) .. controls (554.5,131.38) and (600.25,133.13) .. (600,115) ;
        \draw  [color={rgb, 255:red, 0; green, 0; blue, 255 }  ,draw opacity=1 ] (575.84,102.67) -- (580.99,104.83) -- (576.17,107.66) ;
        \draw  [draw opacity=0] (299.87,171.35) .. controls (285.19,176.85) and (268.16,180) .. (250,180) .. controls (194.77,180) and (150,150.9) .. (150,115) .. controls (150,79.1) and (194.77,50) .. (250,50) .. controls (268.28,50) and (285.41,53.19) .. (300.15,58.75) -- (250,115) -- cycle ; \draw   (299.87,171.35) .. controls (285.19,176.85) and (268.16,180) .. (250,180) .. controls (194.77,180) and (150,150.9) .. (150,115) .. controls (150,79.1) and (194.77,50) .. (250,50) .. controls (268.28,50) and (285.41,53.19) .. (300.15,58.75) ;
        \draw    (185,115) .. controls (186.29,120) and (202.29,120.86) .. (205,115) ;
        \draw    (189.94,118.08) .. controls (192.37,114.08) and (197.51,114.37) .. (199.94,118.08) ;
        \draw    (300.15,98.75) .. controls (261,99.5) and (255.67,130.17) .. (299.87,131.35) ;

        \draw (495,115) node    {$\dotsc $};
        \draw (602,115) node [anchor=west] [inner sep=0.75pt]  [color={rgb, 255:red, 0; green, 0; blue, 255 }  ,opacity=1 ]  {$\gamma$};
        \draw (235,115) node    {$\dotsc $};
        \draw (302.15,55.35) node [anchor=south west] [inner sep=0.75pt]  [color={rgb, 255:red, 0; green, 0; blue, 255 }  ,opacity=1 ]  {$\gamma_{1}$};
        \draw (301.87,174.75) node [anchor=north west][inner sep=0.75pt]  [color={rgb, 255:red, 0; green, 0; blue, 255 }  ,opacity=1 ]  {$\gamma_{2}$};
        \draw (432.45,63.33) node [anchor=south east] [inner sep=0.75pt]    {$S$};
        \draw (170.15,68) node [anchor=south east] [inner sep=0.75pt]    {$S'$};
    \end{tikzpicture}
    \caption{The setup for the bending construction.}
    \label{fig:SS'}
\end{figure}
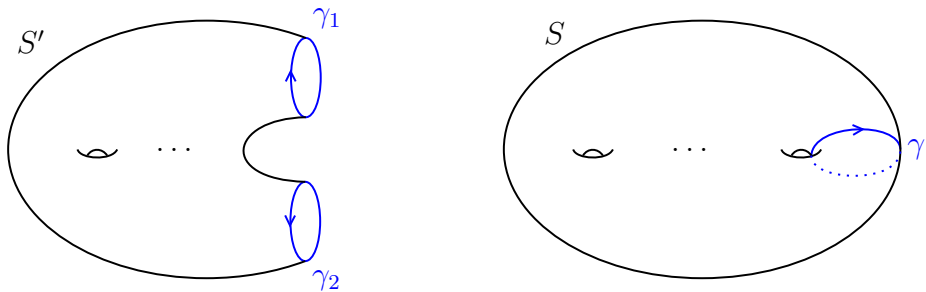

In $S$, the curves $\gamma$, $\gamma_1$ and $\gamma_2$ will all be freely homotopic, and by an abuse of notation, we will write $\gamma_1$ and $\gamma_2$ to denote the elements of $\pi_1(S')$ and identify $\gamma$ with $\gamma_1$ in $\pi_1(S)$. In this setup, it is known that $\pi_1(S)$ can be written as the following \textit{HNN extension}:
\[
\pi_1(S)=\pi_1(S')*_\alpha :=\langle \pi_1(S'),\alpha \,|\,\alpha \gamma_2\alpha ^{-1}=\gamma_1\rangle.
\]
Given any representation $\rho:\pi_1(S)\to\SL(n,\R)$ and any matrix $A\in\SL(n,\R)$ in the centralizer of $\rho(\gamma)$, we can define a new \textit{bent} representation $\rho^A:\pi_1(S)\to\SL(n,\R)$ by setting:
\[
\rho^A|_{\pi_1(S')}:=\rho|_{\pi_1(S')}\quad\textrm{and}\quad\rho^A(\alpha):=A\rho(\alpha).
\]
With extra care for the choice of centralizing $A$, this bending procedure will also preserve a number of useful properties of the original representation.

Firstly, if $\rho$ has image in $\SL(n,\O_K)$ and we choose $A$ in the $\O_K$-centralizer of $\rho(\gamma)$, then it's clear from the formula of the bend that $\rho^A$ still has image inside of $\SL(n,\O_K)$. So with additional care, the bending construction can be done in such a way as to preserve $K$-integrality of the original representation.

Secondly, if we take care to choose $A$ so that $A=\exp(X)$ for some \textit{diagonalizable} $X\in\mathfrak{sl}(n,\R)$, then there is a one-parameter family of matrices given by
\[
A_t:=\exp(tX)\in\SL(n,\R),
\]
for $t\in\R$. For this one-parameter family, we have that $A_0=I_n$, $A_1=A$. Moreover, by diagonalizability, for any $t\in\R$ $A_t$ still centralizes $\rho(\gamma)$. Note that we are using the fact that $\rho(\gamma)$ is also diagonalizable as $\rho$ lies on the Hitchin component to guarantee the matrices $A_t$ all still centralize $\rho(\gamma)$. But in particular, this one-parameter family of matrices $A_t$ defines a path of representations $\rho^{A_t}:\pi_1(S)\to\SL(n,\R)$ from $\rho^{A_0}=\rho$ to $\rho^{A_1}=\rho^A$. Thus, $\rho^A$ lies on the same path component of $\mathcal{R}_n(S)$ as $\rho$. In particular, as $\rho$ is a Hitchin representation, so is $\rho^A$.

So in short, many of the desired properties of the original representation can be preserved by careful choice of the bending matrix $A$, used to deform $\rho$. However, bending construction in this context is useful not just for its ability to preserve ``nice'' properties of our representations, but to ``improve'' them in the sense of producing representations with larger Zariski closures. What ultimately enables this is the following.

\begin{lemma}
\label{largerclosure}
Suppose $\rho(\pi_1(S))$ has Zariski closure contained in $G\neq\SL(n,\R)$. If in addition to the above properties, we choose $A\notin G$, then the Zariski closure of $\rho^A(\pi_1(S))$ is strictly larger than $G$.
\end{lemma}
\begin{proof}
The Zariski closure of $\rho^A(\pi_1(S))$ must still be among the ones of Theorem \ref{guichard} as $\rho^A$ is still on the Hitchin component of $S$. A priori, it may be unclear why the image of $\rho^A$ might not be some nontrivial conjugate of $G$ or a proper subgroup of it, but we can rule each of these possibilities out as well, thus guaranteeing that the Zariski closure must indeed get larger.

As $G\neq\SL(n,\R)$, consulting the list of Theorem \ref{guichard}, we see that either $G$ is conjugate to either $\tau_n(\SL(2,\R))$, $\Sp(2k,\R)$ or $\SO(k+1,k)$. In the first case, $\rho^A(\pi_1(S))$ can't have a smaller Zariski closure as no proper subgroup of $\tau_n(\SL(2,\R))$ is a possible Zariski closure for a Hitchin representation. In the latter two cases, it follows that there is some alternating or symmetric bilinear form, $J$, which is preserved by the image of $\rho$. Basic calculations verify that image of $\SL(2,\R)$ under the irreducible representation $\tau_n$ is contained in $\Sp(2k,\R)$ in the $n$ even case and $\SO(k+1,k)$ in the $n$ odd case. Therefore, if $\rho^A(\pi_1(S))$ had Zariski closure smaller than $G$, then the it must be $\tau_n(\SL(2,\R))$ in which case, the same form $J$ is still preserved. But this cannot be as $A$ won't preserve $J$ by virtue of not being in $G$.

All that remains is to rule out why the Zariski closure of $\rho^A(\pi_1(S))$ can't be some nontrivial conjugate of $G$. If this were the case, then there would be another non-degenerate (symmetric or alternating) bilinear form $J'$ which is preserved by the image of $\rho^A(\pi_1(S))$. But now, observe that the representations $\rho$ and $\rho^A$ are identical on the compact subsurface $S'\subset S$. In particular, the image of the representation $\rho|_{\pi_1(S')}=\rho^A|_{\pi_1(S')}$ preserves both the forms $J$ and $J'$.

As both $J$ and $J'$ are preserved, $J^{-1}J'\rho(\beta)=\rho(\beta)J^{-1}J'$ for all $\beta\in\pi_1(S')$. The subgroup $\pi_1(S')$ has Zariski closure at least containing an (absolutely irreducible) principal $\SL(2,\R)$. $J^{-1}J'$ centralizes all of this principal $\SL(2,\R)$, hence by Schur's lemma, is a homothety. In particular, $J$ and $J'$ are scalar multiples of one another and define the same orthogonal or symplectic groups.

Thus the Zariski closure of $\rho^A(\pi_1(S))$ cannot be a subgroup of, or a conjugate of $G$, and thus must be strictly larger among the possibilities provided by Guichard's list.
\end{proof}

Theorem \ref{maintheorem} is then proven by repeated applications of the bending constructions as follows.

\begin{proof}[Proof of \ref{maintheorem}.]
We let $\rho:\pi_1(S)\to\SL(n,\O_K)$ be a representation on the Hitchin component with Zariski closure $G\neq\SL(n,\R)$. By the various methods in \S\ref{section:constructingdelta} we can construct a matrix $A\in\SL(n,\O_K)$ satisfying the following properties:
\begin{enumerate}[label=(\alph*)]
    \item $A$ centralizes $\rho(\gamma)$ for some nonseparating simple closed curve $\gamma$ (after possibly passing to a finite sheeted cover $\widetilde{S}\to S$);
    \item $A\notin G$;
    \item $A=\exp(X)$ for some diagonalizable $X\in\mathfrak{sl}(n,\R)$.
\end{enumerate}
From such a matrix, property (a) allows us to form the new bent representation $\rho^A$. By property (b), along with Lemma \ref{largerclosure}, this also guarantees that $\rho^A(\pi_1(\widetilde{S}))$ has a larger Zariski closure than $G$. $\rho^A$ still remains on the Hitchin component of $\widetilde{S}$ by property (c) as outlined at the beginning of this section.

We can then repeat this process (at most twice, when $n\neq 7$) until all of the finitely many possible Zariski closures of Theorem \ref{guichard} are ruled out except for $\SL(n,\R)$ itself at which point, the final representation must be Zariski dense, $K$-integral and on the Hitchin component. 
\end{proof}

\section{Constructing the bending matrices}
\label{section:constructingdelta}
To complete the proof of Theorem \ref{maintheorem}, it remains to be shown how one may construct the matrices $A$ satisfying each of properties (a)--(c) used in the proof. Different techniques for constructing such a matrix are needed depending on the various possibilities of the initial Zariski closure provided by Theorem \ref{guichard} as well as the properties of the number field $K$. Before we discuss the specific methods used to construct $A$ in each case, we record the following pair of useful lemmas here which will be used throughout.

\begin{lemma}
\label{integralpower}
Given $M\in\SL(n,K)$ whose characteristic polynomial is in $\O_K[t]$, then there is some $j>0$ so that $M^j\in\SL(n,\O_K)$.
\end{lemma}
\begin{proof}
Using the rational canonical form of $M$, there is a $P\in\GL(n,K)$ so that $M=PRP^{-1}$ where $R\in\SL(n,\O_K)$. Already here, we are implicitly using the hypothesis that $K$ have class number one to guarantee that $R\in\SL(n,\O_K)$. Namely, if $f\in\O_K[t]$ denotes the characteristic polynomial of $M$, then a factorization of $f$ into irreducibles over $\O_K$ will be a factorization into irreducibles over $K$ by Gauss's lemma as $\O_K$ is a UFD. Thus the invariant factors of $f$ can be guaranteed to all be in $\O_K[t]$ as well and so the rational canonical form of $f$ will have entries in $\O_K$.

For any $N\in\O_K$, we consider the reduction map $\SL(n,\O_K)\to\SL(n,\O_K/(N))$. $\SL(n,\O_K/(N))$ is finite so there is some $j$ so that $R^j\equiv I\pmod{(N)}$. In other words, $R^j=I+NX$ for some $X\in\operatorname{M}_n(\O_K)$. In this case, we have that
\[
M^j=P(I+NX)P^{-1}=I+NPXP^{-1}.
\]
In the above expression, the only term depending on $N$ (other than $N$ itself) is $X$, which already has integral entries. Thus, if we pick $N\in\O_K$ so that it clears all the denominators of all entries in $P$ and $P^{-1}$, we get that $M^j\in\SL(n,\O_K)$. 
\end{proof}

\begin{lemma}
\label{irreduciblecharpoly}
Suppose $\rho(\gamma)\in\SL(n,\O_K)$ has characteristic polynomial $f(t)$ where $f\in\O_K[t]$ is irreducible over $\O_K$. Then there is some $A\in\SL(n,\O_K)$ centralizing $\rho(\gamma)$, not preserving the form $J$ and so that $A=\exp(X)$ for some diagonalizable $X\in\mathfrak{sl}(n,\R)$.
\end{lemma}
\begin{proof}
As $f$ is irreducible over the UFD $\O_K$, it is irreducible over $K$. Let $\alpha\in\overline{K}$ denote some root of $f$ and consider the extension $K(\alpha)/K$. Note that $[K(\alpha):K]=n$ as $f$ must be the minimal polynomial of $\alpha$. All the roots of $f$ must be distinct and real as $\rho$ is on the Hitchin component (\ref{labourie:purelyloxodromic}) and so $K(\alpha)$ has at least $n$ real embedding as we may extend the (fixed) real embedding of $K$ to one of $K(\alpha)$ by sending $\alpha$ to any of the other real roots of $f$. As a consequence, the unit group $\O_{K(\alpha)}^\times$ has rank $\geq n-1$.

By diagonalizing $\rho(\gamma)$ over $\R$ and using the fact that it has distinct real eigenvalues, one sees that the centralizer of $\rho(\gamma)$ in $\SO(J;\R)$ has rank $\frac{n}{2}$. In comparison, $\O_{K(\alpha)}^\times$ has rank $\geq n-1$. Then $n-1>\frac{n}{2}$ as long as $n>2$.

Note then that there is a $K$-algebra isomorphism
\[
K(\alpha)\cong K[\rho(\gamma)]
\]
sending $\alpha\mapsto\rho(\gamma)$ because both of the above algebras are isomorphic to $K[t]/(f)$ for the monic irreducible $f\in\O_K[t]$. Consider now, the image of $\O_{K(\alpha)}^\times$ in $K[\rho(\gamma)]$. By the above rank considerations, there is some infinite order $u\in\O_{K(\alpha)}^\times$ whose image, $A'\in K[\rho(\gamma)]$, satisfies the property that no power of $A'$ preserves the form $J$. Further assume that $u>0$ by replacing $u$ with $-u$ if necessary. $A'$ is a polynomial in powers of $\rho(\gamma)$ and so $A'$ itself will still centralize $\rho(\gamma)$. 

At the moment, $A'$ still has some possibly undesirable properties, but we will modify it to construct the matrix $A$. First, note $A'$ has characteristic polynomial in $\O_K[t]$: its characteristic polynomial factors as $\prod_{i=1}^n(t-u^{\sigma_i})$ where the $\sigma_i:K(\alpha)\to\C$ are the $K$-embeddings of $K(\alpha)$. The determinant of $A'$ at the moment is $v=\operatorname{Norm}_{K(\alpha)/K}(u)\in\O_K^\times$, but we may pass to a higher power of $A'$ and rescale by some power of $v^{-1}$ so that $A'$ has determinant $1$. This power of $A'$ then has characteristic polynomial in $\O_K[t]$ and is in $\SL(n,K)$, hence by Lemma \ref{integralpower}, some further suitable power will have entries in $\O_K$. Taking $A=(A')^j$ to be such a large enough power, we get $A\in\SL(n,\O_K)$.

This matrix $A=(A')^j\in\SL(n,\O_K)$ still centralizes $\rho(\gamma)$ and does not preserve the form $J$ as no power of $A'$ preserves $J$ by construction. Furthermore, such a $A$ will be the exponential of some diagonalizable matrix as we arranged for $u>0$ and as $A'$ is a polynomial in powers of $\rho(\gamma)$, all of which are the exponential of some diagonal matrix by Theorem \ref{labourie:purelyloxodromic}.
\end{proof}
\begin{remark}
In the odd-dimensional case, the characteristic polynomials of the $\rho(\gamma)$ are never irreducible over $\O_K$ as they always have an eigenvalue of $1$, due to being on the Hitchin component. Nonetheless, the same conclusion of this lemma can easily be adopted to the odd-dimensional case if we find $\rho(\gamma)$ with characteristic polynomial of the form $(t-1)f(t)$ for some $f(t)\in\O_K[t]$ irreducible over $\O_K$ since we may conjugate $\rho(\gamma)$ to be in block diagonal form with an $(n-1)\times(n-1)$-block along the diagonal. In this case, applying this result to this block gives the bending matrix $A$ in the odd-dimensional case.
\end{remark}

This lemma provides sufficient conditions for which the bending matrix $A$ used in the proof of Theorem \ref{maintheorem} will exist. This implies that to construct a bending matrix, it suffices to find a nonseparating simple closed curve $\gamma\in\pi_1(S)$ whose characteristic polynomials are of the form $f(t)$ or $(t-1)f(t)$ for some irreducible $f(t)$ in $\O_K[t]$. When such $\gamma$ cannot be found, we then use the reducibility of the characteristic polynomials to explicitly construct bending matrices we may use as well. The methods used in constructing these matrices vary depending on what the input to the bend is, particularly, what the number field $K$ is, and what the Zariski closure of the initial representation $\rho(\pi_1(S))$ is.

\subsection{Bending when \texorpdfstring{$K\neq\Q$}{K!=Q}}
\label{section:numberfieldbend}
The simplest case of the bending procedure occurs when $K$ is a proper extension of $\Q$ with class number one. In this case, we are able to prove the following strengthening of Theorem \ref{maintheorem}.

\begin{theorem}
\label{maintheorem:numberfield}
If $K\neq \Q$ and $\rho:\pi_1(S)\to\SL(n,\O_K)$ is a representation on the Hitchin component whose image is not Zariski dense, then for any nonseparating simple closed curve $\gamma\in\pi_1(S)$, there is an $A\in\SL(n,\O_K)$ centralizing $\rho(\gamma)$ so that the bent representation $\rho^A$ is on the Hitchin component and has Zariski dense image.
\end{theorem}

In comparison to Theorem \ref{maintheorem}, when $K\neq \Q$, we are able to ensure the representation is Zariski dense while circumventing the need to pass to any finite sheeted covers. The class number one hypothesis enables various integrality properties to be preserved, but the proof of this result is ultimately enabled by the presence of infinite order units in $\O_K$ allowing one to construct the bending matrix $A$ explicitly.

\begin{proof}[Proof of \ref{maintheorem:numberfield}]
For any nonseparating simple closed $\gamma\in\pi_1(S)$, we consider the characteristic polynomial $f$, of $\rho(\gamma)$. If $f$ is irreducible, then we take the matrix $A$ to be the one guaranteed by Lemma \ref{irreduciblecharpoly}. As $A$ doesn't preserve the (unique) form $J$ which $\rho(\pi_1(S))$ must preserve (as in Lemma \ref{largerclosure}), then $\rho^A(\pi_1(S))$ must have Zariski closure larger than $\SO(J;\R)$ or $\Sp(J;\R)$ and so it must be Zariski dense by Theorem \ref{guichard}.

Thus we may assume that $f$ is reducible over $\O_K$. It suffices to assume that $f=f_1 f_2$ for $f_i\in\O_K[t]$ of degree $n_1,n_2\geq 1$ respectively. Replacing $f_1$ and $f_2$ with $uf_1$ and $u^{-1}f_2$ for some unit $u\in\O_K^\times$ if necessary, we may also assume that $f_1$ and $f_2$ are monic and have constant terms $(-1)^{n_1}$ and $(-1)^{n_2}$ respectively. From the factorization of the characteristic polynomial, there is some matrix $P\in\GL(n,K)$ so that
\[
P\rho(\gamma)P^{-1}=\begin{pmatrix} C_1 \\ &C_2\end{pmatrix}
\]
where for $i=1,2$, $C_i\in\SL(n_i,\O_K)$ has characteristic polynomial $f_i$. Here, we are using the fact that the factorization of $f$ is over $\O_K$ and that $K$ has class number one to guarantee that the entries of the $C_i$ are indeed in $\O_K$.

Now, as $K\neq\Q$, $[K:\Q]>1$. Assuming $K\subset\R$ also forces $K$ to have $\geq 1$ real embeddings, and together, this forces $\O_K^\times$ to have rank $\geq 1$. We may thus fix an infinite order unit $u\in\O_K^\times$ and, replacing $u$ with $-u$ if necessary, we may also assume that $u>0$.

In this case, we take the matrix
\[
A':=P^{-1}\begin{pmatrix}
u^{n_2}I_{n_1}\\& u^{-n_1}C_2
\end{pmatrix}P.
\]
$A'\in\SL(n,K)$ has characteristic polynomial still in $\O_K[t]$ and clearly centralizes $\rho(\gamma)$ due to $PA' P^{-1}$ having the same block-diagonal structure as $P\rho(\gamma)P^{-1}$. As $u>0$ and $C_2$ is diagonalizable, $A'=\exp(X)$ for some diagonalizable $X\in\sl(n,\R)$ as well. Note as well that $A'$ does not preserve the form $J$ as its eigenvalues are $u^{n_2}$ (with multiplicity $n_1$) and $u^{-n_1}\lambda$ where $\lambda$ is any of the (distinct, real) eigenvalues of $C_2$, and these eigenvalues do not pair up as needed for any matrix preserving the form $J$.

Passing to a higher power of $A'$ by Lemma \ref{integralpower}, $A=(A')^j\in\SL(n,\O_K)$ for sufficiently large $j$ still centralizes $\rho(\gamma)$ and remains the exponential of some diagonalizable matrix. As $u$ is an infinite order unit, $A$ will still fail to preserve the form $J$ as $A'$ did. Therefore, by Lemma \ref{largerclosure}, the bent representation $\rho^A$ has Zariski dense image.
\end{proof}

By handling the $K\neq\Q$ case separately using the presence of infinite order units in $\O_K$, we may then assume that $K=\Q$ and $\O_K=\Z$ for the remainder of this section.

\subsection{Bending out of a principal \texorpdfstring{$\SL(2,\R)$}{SL(2,R)}}
When $K=\Q$, no infinite order units are present and so an alternate method is necessary. The overall construction of the bending matrices here is then a two-fold process: one to bend the representation out of a principal $\SL(2,\R)$ and a second to bend the representation out of either $\Sp(2k,\R)$ or $\SO(k+1,k)$ to finally get a representation with Zariski dense image in $\SL(n,\R)$. Different methods are used for each bend, the simpler one when the Zariski closure of $\rho(\pi_1(S))$ is ``as small as possible.''

By Theorem \ref{guichard}, the smallest possible Zariski closure that our representation might have is that of a principal $\SL(2,\R)$ inside of $\SL(n,\R)$. So suppose that $\rho:\pi_1(S)\to\SL(n,\Z)$ has Zariski closure conjugate to $\tau_n(\SL(2,\R))$. In this case, we start by fixing any nonseparating simple closed curve $\gamma\in\pi_1(S)$. 

By Theorem \ref{labourie:purelyloxodromic}, $\rho(\gamma)$ is diagonalizable with distinct real eigenvalues. Moreover, as $\rho(\pi_1(S))$ is contained in some principal $\SL(2,\R)$, basic properties about the irreducible representations of $\SL(2,\R)$ tell us that the eigenvalues of $\rho(\gamma)$ are of the form $\lambda^{n-2i-1}$ for $i=0,\ldots,n-1$. In this case, the characteristic polynomial of $\rho(\gamma)$ is of the form $(t-1)f(t)$ or $f(t)$ for some $f\in\Z[t]$ depending on the parity of $n$. If $f(t)$ is irreducible in either of the above cases, we again apply Lemma \ref{irreduciblecharpoly} to construct the matrix $A$ (in the $n$ odd case, we proceed as in the remark after Lemma \ref{irreduciblecharpoly}). Consequently, the matrix $A$ provided by this theorem does not preserve the form this principal $\SL(2,\R)$ does, and so $\rho^A$ will then have Zariski dense image. So just as with the $K\neq \Q$ construction, we are able to guarantee the representation has Zariski dense image after a single bend, \textit{if} we can find a $\gamma\in\pi_1(S)$ with this property.

What remains in this case is when $f(t)$ is reducible over $\Z$. That is, we may assume the characteristic polynomial of $\rho(\gamma)$ admits a factorization of the form
\[
(t-1)f_1(t)f_2(t)\quad\textrm{or}\quad f_1(t)f_2(t),
\]
for polynomials $f_1,f_2\in\Z[t]$ of degree $n_1,n_2\geq 1$ respectively. From the factorization of the characteristic polynomial, it follows then that there is some matrix $P\in\GL(n,\Q)$ so that
\[
P\rho(\gamma)P^{-1}=\begin{cases}
\displaystyle\begin{pmatrix}
1 \\
& C_1\\
&&C_2\end{pmatrix} & \textrm{if $n$ is odd}\\
\displaystyle\begin{pmatrix} C_1\\&C_2\end{pmatrix} & \textrm{if $n$ is even,}
\end{cases}
\]
where for $i=1$ or $2$, $C_i\in \SL(n_i,\Z)$ has characteristic polynomial $f_i$. We now take $A'$ to be the matrix
\[
A'=\begin{cases}
\displaystyle P^{-1}\begin{pmatrix} 1\\&I_{n_1}\\&&C_2\end{pmatrix}P &\textrm{if $n$ is odd}\\
\displaystyle P^{-1}\begin{pmatrix} I_{n_1}\\&C_2\end{pmatrix} P&\textrm{if $n$ is even.}
\end{cases}
\]
Note that in both cases, $\det(A')=\det(C_2)=1$ and $A'$ is in the $K$-centralizer of $\rho(\gamma)$.

By Lemma \ref{integralpower} there is some $j>0$ so that $(A')^j$ has entries in $\Z$. We take $A=(A')^j$. Then by construction, $A\in\SL(n,\Z)$ centralizes $\rho(\gamma)$. Moreover, $A$ is not contained in a principal $\SL(2,\R)$ because the eigenvalues of $A$ are not of the form $\mu^{n-2i-1}$ for some $\mu\in\R$ since $A$ has an eigenvalue $1$ of multiplicity $>1$. Finally, we show that $A$ is in the image of the exponential map. For this, it suffices to show $A'=\exp(X)$ for some $X\in\mathfrak{sl}(n,\R)$ since then, $A=(A')^j=\exp(jX)$. In this case, if we let $\lambda_1,\ldots,\lambda_{n_2}$ be the distinct positive real eigenvalues of $C_2$, then $C_2$ is diagonalizable over $\R$ and we have that
\[
A'=Q^{-1}\exp\begin{pmatrix}0\\&\ddots\\&&0\\&&&\log(\lambda_1)\\&&&&\ddots\\&&&&&\log(\lambda_{n_2})\end{pmatrix}Q
\]
for some $Q\in\GL(n,\R)$. Moreover, $\det(C_2)=1$ so that the trace of the above matrix is $0$, hence $A'=\exp(X)$ for some diagonalizable $X\in\mathfrak{sl}(n,\R)$.

As a result, the matrix $A$ centralizing $\rho(\gamma)$ can be used to bend the representation $\rho$ and produce a representation $\rho^A:\pi_1(S)\to\SL(n,\Z)$ on the Hitchin component with Zariski closure strictly larger than a principal $\SL(2,\R)$. At this point, the resulting bent representation may still have Zariski closure conjugate to $\SO(k+1,k)$ or $\Sp(2k,\R)$ when $n=2k+1$ or $n=2k$ respectively. Thus another bend is still necessary to guarantee that one may produce a final representation with Zariski dense image.

\subsection{Bending out of \texorpdfstring{$\Sp(2k,\R)$}{Sp(2k,R)}}
\label{Spbend}
We handle the even case first which follows the same general strategy as the bend out of $\SO(k+1,k)$, but is practically easier, due to facts about symplectic groups which simplify various steps. Assume that $\rho:\pi_1(S)\to\SL(2k,\Z)$ has Zariski closure $\Sp(2k,\R)$. The starting point is the following Strong Approximation result from the theory of algebraic groups.

\begin{theorem}[Matthews et al., \cite{matthewsetal}]
\label{strongapprox}
Let $G$ be a connected simply-connected absolutely almost simple algebraic group defined over $\Q$ and $\Gamma\leqslant G(\Q)$ a finitely generated Zariski dense subgroup. 

Then for all but finitely many primes $p$, the reduction map
\[
\pi_p:\Gamma\to G(\F_p)
\]
is onto.
\end{theorem}

We will also need the following result of Borel's.

\begin{theorem}[Borel]
\label{symplecticcharpoly}
Let $p>4$ and let $R_p(2k)$ be the set of $2k\times 2k$ matrices in $\Sp(2k,\F_p)$ with reducible characteristic polynomial. Then
\[
|R_p(2k)|\leq \left(1-\frac{1}{3k}\right)|\Sp(2k,\F_p)|.
\]
\end{theorem}

By Theorem \ref{strongapprox} for $G=\Sp(2k)$ and $\Gamma=\rho(\pi_1(S))$, we obtain a surjection 
\[
\pi_p\circ\rho:\pi_1(S)\twoheadrightarrow\Sp(2k,\F_p),
\]
for some prime $p$ which, when combined with Theorem \ref{symplecticcharpoly}, implies there exists some loop $\eta\in\pi_1(S)$ so that the characteristic polynomial of $\pi_p(\rho(\eta))$ is irreducible modulo $p$. Consequently, $\rho(\eta)$ has characteristic polynomial irreducible over $\Z$. Now, at the moment, this loop may not be simple, but the following result allows us to remedy this.

\begin{theorem}[Scott, \cite{scottlifting}]
\label{scottsliftingthm}
Every closed curve on $S$ lifts to a simple closed curve in some finite-sheeted cover.
\end{theorem}

We now pass to a finite sheeted cover $\widetilde{S}\to S$ where $\eta$ lifts simply and restrict our representation $\rho$, to the subgroup $\pi_1(\widetilde{S})\leqslant\pi_1(S)$. By passing to a further finite sheeted cover if needed, we may also assume that the loop $\eta\in\pi_1(\widetilde{S})$ is nonseparating. The following lemma is needed before then performing the bend out of $\Sp(2k,\R)$.

\begin{lemma}
\label{samezariskiclosure}
The subgroup $\rho(\pi_1(\widetilde{S}))\leqslant\rho(\pi_1(S))$ has Zariski closure $\Sp(2k,\R)$.
\end{lemma}
\begin{proof}
Let $H$ denote the Zariski closure of $\rho(\pi_1(\widetilde{S}))$. Clearly $H\leqslant\Sp(2k,\R)$, but as $\widetilde{S}$ is still a closed surface and $\rho|_{\pi_1(\widetilde{S})}$ lies on the Hitchin component of $\widetilde{S}$, Guichard's classification (\ref{guichard}) implies that either $H=\Sp(2k,\R)$ or $H$ is a principal $\SL(2,\R)$.

But $H$ is of finite index in $G$, so if $g_1,\ldots,g_k$ are finitely many coset representatives for $G/H$, then $g_1 H\cup\ldots\cup g_kH$ is a Zariski-closed subset containing $G$. If $H$ was a principal $\SL(2,\R)$, this would imply that $\Sp(2k,\R)$ is a union of finitely many principal $\SL(2,\R)$'s and simply for dimension reasons this is even impossible.
\end{proof}

By this lemma, we have a representation $\rho:\pi_1(\widetilde{S})\to\SL(n,\Z)$ which is Zariski dense in $\Sp(2k,\R)$ and on the Hitchin component of $\widetilde{S}$. Moreover, the curve $\eta\in\pi_1(S)$ lifts simply in $\widetilde{S}$ and the characteristic polynomial of $\eta$ is irreducible over $\Z$. In light of Lemma \ref{irreduciblecharpoly}, there is some bending matrix $A\in\SL(n,\Z)$ so that the bent representation $\rho^A$ is Zariski dense in $\SL(n,\R)$.

\subsection{Bending out of \texorpdfstring{$\SO(k+1,k)$}{SO(k+1,k)}}
\label{SObend}
In the odd $n$ case, the only possible representations remaining are the $\Q$-integral ones with Zariski closure conjugate to $\SO(k+1,k)$. Suppose that $\rho:\pi_1(S)\to\SL(2k+1,\Z)$ has Zariski closure $\SO(J;\R)$ where $J$ is some symmetric bilinear form of signature $(k+1,k)$. In this case, we prove the following fact.

\begin{lemma}
\label{SObendlemma}
There exists a $\eta\in\pi_1(S)$ so that $\rho(\eta)$ has characteristic polynomial $(t-1)f(t)\in\Z[t]$ with $f(t)$ irreducible over $\Z$.
\end{lemma} 

The required loop is constructed in essentially the same manner as done in \S $3.2$ of \cite{longthis:2020}, but we recap the construction, and note the minor differences here. Again, the first step will be to look at the reduction map $\SO(J,\Z)\to\SO(J,\F_p)$ for some prime $p$. In this case, over odd-dimensional vector spaces over finite fields, there is a unique orthogonal group up to isomorphism, dependent only on the dimension and the prime $p$ (Theorem $5.8$ in \cite{suzukigroups}). Thus we will make the identification $\operatorname{O}(J,\F_p)=\operatorname{O}(2k+1,p)$. In this case, we let $\SO(2k+1,p)$ denote the (unique up to isomorphism) special orthogonal group over $\F_p$ and also set $\Omega(2k+1,p)=[\operatorname{O}(2k+1,p),\operatorname{O}(2k+1,p)]\leqslant\SO(2k+1,p)$ to be the commutator subgroup. $\Omega(2k+1,p)$ is a simple subgroup of index $4$ in $\operatorname{O}(2k+1,p)$ (cf. \cite{suzukigroups}, p. 383--384).

Just as in \S\ref{Spbend}, the curve in Lemma \ref{SObendlemma} with the required characteristic polynomial will be constructed by looking at the characteristic polynomials modulo some prime $p$. However, we can't simply apply Strong Approximation as is stated in Theorem \ref{strongapprox} simply due to the fact that the algebraic group $\SO(J)$ is not simply connected. Instead, the following corollary of Strong Approximation (cf. \cite{weisfeiler}) still provides us with enough in the image of the reduction map to construct the needed $\eta$.

\begin{theorem}
For all but finitely many primes $p$, the image of the composition
\[
\pi_p\circ\rho:\pi_1(S)\to\SO(2k+1,p)
\]
contains the subgroup $\Omega(2k+1,p)$.
\end{theorem}

Lemma \ref{SObendlemma} then follows immediately from the previous theorem and the following result whose proof is Theorem $3.8$ of \cite{longthis:2020}.

\begin{lemma}
For every prime $p$, there is a matrix in $\Omega(2k+1,p)$ with a characteristic polynomial of the form $(t-1)\overline{f}(t)$ where $\overline{f}$ is irreducible modulo $p$.
\end{lemma}

By these last two results, we may find some $\eta\in\pi_1(S)$ so that $\pi_p(\rho(\eta))$ has characteristic polynomial of the form $(t-1)\overline{f}(t)$ for some $\overline{f}$ irreducible modulo $p$. It follows that $\rho(\eta)$ has characteristic polynomial of the form $(t-1)f(t)$ for some $f$ irreducible over $\Z$. As before, $\eta$ may not be simple, but Theorem \ref{scottsliftingthm} again lets us pass to a finite sheeted cover $\widetilde{S}\to S$ where $\eta$ lifts simply to a nonseparating curve. Again, as in Lemma \ref{samezariskiclosure}, $\rho(\pi_1(\widetilde{S}))\leqslant\rho(\pi_1(S))$ also has Zariski closure $\SO(J;\R)$ and so we may perform the bend on the restriction $\rho:\pi_1(\widetilde{S})\to\SL(n,\Z)$, along the simple closed curve $\eta$, to get a representation $\rho^A:\pi_1(\widetilde{S})\to\SL(n,\Z)$ with Zariski dense image in $\SL(n,\R)$ via Lemma \ref{irreduciblecharpoly}. This last remaining case then completes the proof of Theorem \ref{maintheorem}.

\section{Applications}
The methods of this paper also have a number of useful consequences and applications. We record some of these here.

\begin{corollary}
If $\Hit_n(S)$ contains an $\O_K$-point, then there are infinitely many $\O_K$-points in $\Hit_n(S)$ which are Zariski-dense.
\end{corollary}
\begin{proof}
This mainly is a consequence of the fact that the construction of the bending matrices $A$ rely on a number of arbitrary choices made, and any difference in these choices lead to different Zariski dense representations of $\pi_1(S)$. For instance, if $A$ is a bending matrix satisfying properties (a)--(c) used in the proof of Theorem \ref{maintheorem}, then so is $A^m$ for any integer $m\geq 1$, and hence $\rho^{A^m}$ are all alternative bends of $\rho$ with Zariski dense image. Pairwise non-conjugacy of this family of Zariski dense representations also follows by an argument similar to the one given in Lemma \ref{largerclosure} using Schur's lemma, leveraging the fact that the family of bends $\{\rho^{A^m}\}_{m=1}^\infty$ all agree on the subgroup of the cut subsurface, $\pi_1(S')$, whose Zariski closure contains a principal $\SL(2,\R)$.

Additionally, in the setting of \S\ref{section:numberfieldbend}, the construction works for \textit{any} nonseparating simple closed curve on $S$, and so for each curve, different Zariski dense surface groups may be produced. Notably, all these representations also lie in different mapping class group orbits in $\Hit_n(S)$ as they all have distinct images.
\end{proof}

\begin{corollary}
The Zariski dense representations in $\Hit_n(S)$ are (classically) dense in the Hitchin component.
\end{corollary}
\begin{proof}
This rests on the fact that the same techniques used in \S\ref{section:numberfieldbend} may be used to perform a bend. However, in this case, when we no longer require anything regarding integrality of the representation to be preserved, we may take bending matrices of the form $\exp(tX)$ for as small of a $t>0$ as we desire. Furthermore, these bends can be performed about any simple closed curve on $S$. 

Thus for any $\rho\in\Hit_n(S)$ and any open neighborhood of $\rho$, we may bend about any simple closed curve using a matrix $\exp(tX)$ for as small $t>0$ as necessary to produce Zariski dense representations near $\rho$.
\end{proof}

Other interesting questions in these directions still remain. For instance, whether or not the representations produced by this construction are non-commensurable (up to conjugacy) seems to be unknown. Part of this question is answered in \cite{longreidcocompact} for representations into cocompact lattices of $\SL(3,\R)$. There, the authors deduce pairwise non-commensurability up to conjugacy of an infinite family of thin surface subgroups by proving that the surface subgroups in this family had projectively distinct limit sets. 

In this context, the associated limit set may be studied in terms of the hyperconvex Frenet curve $\xi_\rho:\partial_\infty\pi_1(S)\to\operatorname{Flag}(\R^n)$ associated to any Hitchin representation $\rho$, first introduced in \cite{Labourie2006}. At the moment, we can show that the family of bends $\{\rho^{A^m}\}_{m=1}^\infty$ produces representations whose flag \textit{curves} are distinct, i.e. the functions $\xi_{\rho^{A^j}}$ and $\xi_{\rho^{A^k}}$ are distinct whenever $j\neq k$. Full non-commensurability up to conjugacy would require showing that these flag curves have distinct \textit{images}, up to the action of $\SL(n,\R)$, which is more subtle and seems to be currently unknown.

We close with some known examples of representations this bending construction can be applied to.

\begin{enumerate}[label=(\arabic*)]
    \item The $(3,4,4)$-triangle group, $\Delta(3,4,4)$, is the fundamental group of the orbifold with signature $S^2(3,4,4)$. Using the presentation $\Delta(3,4,4)=\langle \alpha,\beta,\gamma\,|\,\alpha^3= \beta^4= \gamma^4=\alpha \beta \gamma=1\rangle$, the representation realizing the (unique up to conjugacy) hyperbolic structure on this orbifold is given by:
    \begin{align*}
    \rho(\alpha)&=\begin{pmatrix}
    0 & -1\\1 & \phantom{-}1\end{pmatrix},\\
    \rho(\beta)&=\begin{pmatrix}0 & -1-\sqrt{2}\\-1+\sqrt{2} & \sqrt{2}\end{pmatrix},\\
    \rho(\gamma)&=\begin{pmatrix}1-\sqrt{2} & -\sqrt{2}\\-1+\sqrt{2} & -1\end{pmatrix}.
    \end{align*}
    As this comes from a hyperbolic structure on $\Delta(3,4,4)$, $\rho$ represents a point in $\mathcal{T}(S^2(3,4,4))$. In \cite{longthis:2020} (Theorem $2.1$), the authors show that, when composing with the irreducible representation $\tau_n:\SL(2,\R)\to\SL(n,\R)$, $\tau_n\rho$ can be conjugated to have image inside $\SL(n,\Z)$ when $n$ is odd, or $\SL(n,\Z[\sqrt{2}])$ when $n$ is even (note: $\Z[\sqrt{2}]=\O_{\Q(\sqrt{2})}$ is a PID).
    
    $\Delta(3,4,4)$ contains a torsion-free subgroup of finite-index which is the fundamental group of some closed surface: $\pi_1(S)\leqslant\Delta(3,4,4)$. Restricting this representation of $\Delta(3,4,4)$ to one of $\pi_1(S)$, we see that Hitchin component of $S$ therefore contains a $\Z$-point or a $\Z[\sqrt{2}]$-point depending on the parity of $n$, and so can be bent to being Zariski dense.
    \item Examples of $\Q$-integral even-dimensional Hitchin representations are largely unknown. At the moment of writing, such examples are only known when $n=4$. Long and Thistlethwaite provided such examples of Hitchin representations into $\SL(4,\Z)$ in \cite{longthis:2018}: one coming from an infinite family of representations of the triangle group $\Delta(3,3,4)$, along with one of the triangle group $\Delta(2,4,5)$. These methods then provide an alternate proof of the results of the aforementioned paper, giving Zariski dense surface subgroups in $\SL(4,\Z)$, but additional examples for $\SL(2k,\Z)$ when $k>2$ are still unknown.
\end{enumerate}
\bibliographystyle{amsalpha}
\bibliography{references.bib}
\end{document}